\documentclass{amsart}
\usepackage{amssymb,mathrsfs,MnSymbol}%,skak}
\usepackage{graphicx}
\usepackage{color}

\def\bC {\mathbf{C}}

\def\bE {\mathbf{E}}

\def\bN {\mathbf{N}}

\def\bR {\mathbf{R}}

\def\bT {\mathbf{T}}

\def\fH {\mathfrak{H}}
\def\fS {\mathfrak{S}}

\def\cA {\mathcal{A}}

\def\cD {\mathcal{D}}

\def\cF {\mathcal{F}}

\def\cH {\mathcal{H}}

\def\cL {\mathcal{L}}

\def\cS {\mathcal{S}}

\def\cU {\mathcal{U}}
\def\cV {\mathcal{V}}

\def\a {{\alpha}}
\def\b {{\beta}}
\def\g {{\gamma}}

\def\de {{\delta}}

\def\th {{\theta}}

\def\ka {{\kappa}}
\def\l {{\lambda}}
\def\L {{\Lambda}}
\def\si {{\sigma}}

\def\om {{\omega}}

\def\d {{\partial}}
\def\grad {{\nabla}}
\def\Dlt {{\Delta}}

\def\rstr {{\big |}}
\def\indc {{\bf 1}}

\def\wtilde {\widetilde }

\def\la {\langle}
\def\ra {\rangle}
\def \La {\bigg\langle}
\def \Ra {\bigg\rangle}

\newcommand{\Tr}{\operatorname{trace}}

\newcommand{\Op}{\operatorname{OP}}

\def\hb {{\hbar}}

\newcommand{\ba}{\begin{aligned}}
\newcommand{\ea}{\end{aligned}}

\newcommand{\be}{\begin{equation}}
\newcommand{\ee}{\end{equation}}

\newcommand{\lb}{\label}

\newtheorem{Thm}{Theorem}[section]

\newtheorem{Prop}[Thm]{Proposition}

\newtheorem{Lem}[Thm]{Lemma}
\newtheorem{Def}[Thm]{Definition}

\begin{document}

\title[Random Batch for Quantum Dynamics]{The Random Batch Method\\ for $N$-Body Quantum Dynamics}

\author[F. Golse]{Fran\c cois Golse}
\address[F.G.]{\'Ecole polytechnique, CMLS, 91128 Palaiseau Cedex, France}
\email{francois.golse@polytechnique.edu}

\author[S. Jin]{Shi Jin}
\address[S.J.]{School of Mathematical Sciences, Institute of Natural Sciences, MOE-LSC,  Shanghai Jiao Tong University, Shanghai 200240, China}
\email{shijin-m@sjtu.edu.cn}

\author[T. Paul]{Thierry Paul}
\address[T.P.]{Sorbonne Universit\'e, CNRS, Universit\'e de Paris, INRIA, Laboratoire Jacques-Louis Lions, 75005 Paris, France}
\email{thierry.paul@polytechnique.edu}

\begin{abstract}
This paper discusses a numerical method for computing the evolution of large interacting system of quantum particles. The idea of the random batch method is to replace the total interaction of each particle with the $N-1$ 
other particles by the interaction with $p<N$ particles chosen at random at each time step, multiplied by $(N-1)/p$. This reduces the computational cost of computing the interaction partial per time step from
$O(N^2)$ to $O(N)$. For simplicity, we consider only in this work the case $p=1$ --- in other words, we assume that $N$ is 
even, and that at each time step, the $N$ particles are organized in $N/2$ pairs, with a random reshuffling of the pairs at the beginning of each time step. We obtain a convergence estimate for the Wigner transform of the 
single-particle reduced density matrix of the particle system at time $t$ that is uniform in $N>1$ and independent of the Planck constant $\hbar$.
\end{abstract}

\keywords{Time-dependent Schr\"odinger equations, Random batch method, Mean-field limit, Wasserstein distance}

\subjclass{82C10, 82C22 (65M75)}

\maketitle

%%%%%%%%%%%%%%%%%%%%%%%%%%%%%%%%%%%%%%%%%%%%%%%%%%%%%%%%%%%%%%%%%%%%%%%%%%%%%%%%%%%%%%%%%%%%%%%%%%%%%%%%%%%%%

\section{Introduction}\lb{S-Intro}

%%%%%%%%%%%%%%%%%%%%%%%%%%%%%%%%%%%%%%%%%%%%%%%%%%%%%%%%%%%%%%%%%%%%%%%%%%%%%%%%%%%%%%%%%%%%%%%%%%%%%%%%%%%%%

Consider the quantum Hamiltonian for $N$ identical particles at the positions $x_1,\ldots,x_N\in\bR^d$:
\be\lb{NBodyHam}
\cH_N:=\sum_{m=1}^N-\tfrac12\hbar^2\Dlt_{x_m}+\tfrac1{N-1}\sum_{1\le l<n\le N}V(x_l-x_n)\,.
\ee
The $N$- particles in this system interact via a binary (real-valued ) potential $V$ assumed to be even, bounded and sufficiently regular (at least of class $C^{1,1}$ on $\bR^d$). The coupling constant $\tfrac1{N-1}$ is 
chosen in order to balance the summations in the kinetic energy (involving $N$ terms) and in the potential energy (involving $\tfrac12N(N-1)$ terms). 

We seek to compute the solution $\Psi\equiv\Psi(t,x_1,\ldots,x_N)\in\bC$ of the Schr\"odinger equation
\be\lb{NBodySchro}
i\hb\d_t\Psi(t,x_1,\ldots,x_N)=\cH_N\Psi(t,x_1,\ldots,x_N)\,,\quad\Psi\rstr_{t=0}=\Psi^{in}
\ee
where $t\ge 0$ is the time while $x_m\in\bR^d$ is the position of the $m$th particle. When solving  \eqref{NBodySchro}, the computation
is exceedingly expensive due to the 
smallness of $\hb$ which demands small time steps $\Dlt t$ and small mesh sizes of order $\hb$ for the convergence of the numerical scheme, due to the oscillation in the wave function $\Psi$ with frequency of order $1/\hb$  
(see \cite{BaoJinMarko,JMS}). On top of this, any numerical scheme for \eqref{NBodySchro} requires computing, at each time step, the sum of the interaction potential for each particle pair in the $N$-particle system, i.e. the 
sum of $\tfrac12N(N-1)$ terms. For large values of $N$, the cost of this computation, which is of order $O(N^2)$,  may become significant at each time step. The purpose of the Random Batch Method (RBM) described below 
is precisely to reduce significantly the computational cost of computing the interacting potential from $O(N^2)$ to $O(N)$.

Throughout this paper, we assume for simplicity that $N\ge 2$ is an even integer. Let $\si_1,\si_2,\ldots,\si_j,\ldots$ be a random sequence of mutually independent permutations distributed uniformly in $\fS_N$. Each permutation 
$\si\in\fS_N$ defines a partition of $\{1,\ldots,N\}$ into $N/2$ batches of two indices (pairs) as follows:
$$
\{1,\ldots,N\}=\coprod_{k=1}^{N/2}\{\si(2k-1),\si(2k)\}\,.
$$
Pick a time step $\Dlt>0$, set
$$
\bT_t(l,n):=\left\{\ba{}&1\quad&&\text{ if }\{l,n\}=\left\{\si_{[\frac{t}{\Dlt t}]+1}(2k\!-\!1),\si_{[\frac{t}{\Dlt t}]+1}(2k)\right\}\text{ for some }k=1,\ldots,\tfrac{N}2\,,\\ &0&&\text{ otherwise,}\ea\right.
$$
and consider the time-dependent Hamiltonian
\be\lb{RBHam}
\cH_N(t):=\sum_{m=1}^N-\tfrac12\hbar^2\Dlt_{x_m}+\sum_{1\le l<n\le N}\bT_t(l,n)V(x_l-x_n)\,.
\ee
In other words, at each time step, the particle labels $m=1,\ldots,N$ are reshuffled randomly, then grouped pairwisely,  and the potential applied to the $m$th particle by the system of $N-1$ other particles is replaced with the
interaction potential of that particle with the other --- {\it only one in this case } --- particle in the same group (batch). 

The motivation of the RBM is that the computation of the solution $\wtilde\Psi\in\bC$ of the time-dependent, random batch Schr\"odinger equation
\be\lb{RBSchro}
i\hb\d_t\wtilde\Psi(t,x_1,\ldots,x_N)=\cH_N(t)\wtilde\Psi(t,x_1,\ldots,x_N)\,,\quad\wtilde\Psi\rstr_{t=0}=\wtilde\Psi^{in}
\ee
is much less costly than computing the solution $\Psi_N$ of the $N$-body Schr\"odinger \eqref{NBodySchro} for large values of $N$. Clearly,  for each time step the cost of computing the interaction potential is reduced from 
$O(N^2)$ to $O(N)$. We remark that  the computational cost of reshuffling the $N$ labels is $O(N)$ by Durstenfeld's algorithm \cite{Durstenfeld}. Of course, one needs to prove that (\ref{RBSchro}) is a ``good approximation 
of $\Psi_N$''  for a sufficiently small time-step $\Dlt t$. 

Our goal in the present paper is to show that the RBM converges in some sense as $\Dlt t\to 0$, with an error estimate that is

(a) independent of $N$, and

(b) uniform in $\hb\in(0,1)$.

Obviously, one wishes to use the RBM for finite, albeit possibly large, values of $N$. It is therefore an obvious advantage to have an error estimate for the RBM that is {\it independent of $N$}, instead of an asymptotic rate of
convergence that would be valid only in the limit as $N\to+\infty$. This explains the need for condition (a). Moreover, the RBM is known to converge in the case of classical dynamics (see \cite{SJinLLiJGLiu}). It is therefore
natural to seek an error estimate for the quantum RBM method which does not deteriorate in the semiclassical regime, and this accounts for condition (b).

Our main results on this problem are gathered in the next section.

There are many variants of the RBM presented above. For instance, one could divide the $N$ particles in batches of $p$ (instead of only $2$, but with $p<<N)$ particles to enhance the accuracy, or reduce the variance of the 
method (assuming of course that $N$ is a multiple of $p$ for simplicity). Likewise, one could replace the PDE \eqref{RBSchro} with some numerical approximation thereof --- for instance one could approximate the solution 
of \eqref{RBSchro} by alternating direction method, where, at each time-step, one replaces the resolution of \eqref{RBSchro} by that of $N/2$ Schr\"odinger $2$-body equations for each particle pair belonging to the same batch.

The RBM for the classical dynamics of large particle systems has been proposed and analyzed in \cite{SJinLLiJGLiu}. However, obtaining error bounds which satisfy the conditions (a)-(b) listed above on random batch algorithms 
for quantum particle systems requires completely new ideas, especially on the problem of metrizing the state space as the number $N$ of interacting particles tends to infinity. For that reason, the present paper discusses only the 
simplest possible formulation of the RBM, specifically the approximation of the solution of \eqref{NBodySchro} by that of \eqref{RBSchro} with batches of only $2$ particles, in order to focus our attention on the essential features 
of this problem.

The origins of the random batch method introduced in \cite{SJinLLiJGLiu} for classical interacting particle systems  can be found in stochastic programming (see for instance the discussion of the stochastic gradient method in 
\cite{Nemirovski}), and more specifically in the applications of that method in the context of machine learning (see \cite{BachMoulines,YingYuanVlaskiSayed} and the references therein). An detailed presentation of stochastic 
approximation methods can be found in \cite{Kushner}, while \cite{Benaim} provides a nice introduction to the dynamical aspects of these methods. In \cite{SJinLLiJGLiu}, and also in the problem under study in this paper,  an
error estimate of the RBM is established for unsteady, time-dependent problems, while in stochastic optimization methods such as the stochastic Gradient descent methods, one uses pseudo-time and the goal is to prove
convergence toward the steady state. 

%%%%%%%%%%%%%%%%%%%%%%%%%%%%%%%%%%%%%%%%%%%%%%%%%%%%%%%%%%%%%%%%%%%%%%%%%%%%%%%%%%%%%%%%%%%%%%%%%%%%%%%%%%%%%

\section{Mathematical Setting and Main Result}

%%%%%%%%%%%%%%%%%%%%%%%%%%%%%%%%%%%%%%%%%%%%%%%%%%%%%%%%%%%%%%%%%%%%%%%%%%%%%%%%%%%%%%%%%%%%%%%%%%%%%%%%%%%%%

We have introduced the $N$-body quantum dynamics and its random batch approximation via the Schr\"odinger equations \eqref{NBodySchro} and \eqref{RBSchro}. However, it will be more convenient to couch the analysis
leading to our error estimates in terms of the corresponding von Neumann equations, which we recall below.

Henceforth we denote $\fH:=L^2(\bR^d;\bC)$ and $\fH_N=\fH^{\otimes N}\simeq L^2((\bR^d)^N;\bC)$ for each $N\ge 2$. The algebra of bounded operators on $\fH$ is denoted by $\cL(\fH)$, while $\cL^1(\fH)\subset\cL(\fH)$
and $\cL^2(\fH)$ are respectively the two-sided ideals of trace-class and Hilbert-Schmidt operators on $\fH$. The operator norm of $A\in\cL(\fH)$ is denoted $\|A\|$. A density operator on $\fH$ is a trace-class operator $R$ on 
$\fH$ such that
$$
R=R^*\ge 0\quad\text{ and }\quad\Tr_\fH(R)=1\,.
$$
An example\footnote{Throughout this paper, we use Dirac's bra-ket notation. A square integrable function $\psi\equiv\psi(x)\in\bC$ viewed as a vector in $\fH$ is denoted $|\psi\ra$, while the notation $\la\psi|$ designates
the linear functional
$$
\fH\ni\phi\mapsto\int_{\bR^d}\overline{\psi(x)}\phi(x)dx=:\la\psi|\phi\ra\in\bC\,.
$$}
of density operator on $\fH$ is the $\fH$-orthogonal projection on $\bC\psi$ for $\psi\in\fH$ satisfying $\|\psi\|_\fH=1$, henceforth denoted $|\psi\ra\la\psi|$.

The $N$-body von Neumann equation is the following differential equation with unknown $t\mapsto R(t)$, an operator-valued function of $t$:
\be\lb{NBodyvN}
i\hb\d_tR(t)=\cH_NR(t)-R(t)\cH_N=:[\cH_N,R(t)]\,,\quad R(0)=R^{in}\,.
\ee
Since $V\in C(\bR^d)$ is bounded real-valued, the $N$-body quantum Hamiltonian $\cH_N$ has a self-adjoint extension to $\fH_N$, so that the solution of \eqref{NBodySchro} is $\Psi(t,\cdot)=e^{-it\cH_N}\Psi^{in}$, while the 
solution of \eqref{NBodyvN} is given by 
\be\lb{R=}
R(t)=e^{-it\cH_N/\hb}R^{in}e^{it\cH_N/\hb}\,.
\ee
In particular, if $R^{in}\in\cD(\fH_N)$, then $R(t)\in\cD(\fH_N)$ for each $t\ge 0$. 

Likewise, if $R^{in}=|\Psi^{in}\ra\la\Psi^{in}|$, then $R(t)=|\Psi(t)\ra\la\Psi(t)|$ for each $t\ge 0$. Conversely, if $R(t)$ is a rank-one density operator, its range is of the form $\bC\Psi(t)$ with $\|\Psi(t)\|_{\fH_N}=1$, and this defines 
a unique $\Psi(t)$ so that $R(t)=|\Psi(t)\ra\la\Psi(t)|$ up to multiplication by a complex number of modulus one. In other words, $R(t)$ is in one-to-one correspondence with the quantum state associated to $\Psi(t)$, that is to say, 
in accordance with the Born interpretation, with the complex line in $\fH_N$ spanned by $\Psi(t)$. This explains the connection between \eqref{NBodySchro} and \eqref{NBodyvN}.

Likewise, the random batch von Neumann equation is the differential equation with unknown $t\mapsto\wtilde{R}(t)$, an operator-valued function of $t$:
\be\lb{RBvN}
i\hb\d_t\wtilde R(t)=[\cH_N(t),\wtilde R(t)]\,,\quad \quad\wtilde R(0)=R^{in}\,.
\ee
The formula giving $\wtilde R(t)$ is
$$
\wtilde R(t)=U(t,0)\wtilde R(0)U(0,t)
$$
where, for each $0\le s\le t$,
$$
\ba
{}&U(s,t):=e^{-\frac{i(s-[s/\Dlt t]\Dlt t)}\hb\cH_N([\frac{s}{\Dlt t}]\Dlt t)}\prod_{j=[s/\Dlt t]}^{[t/\Dlt t]-1}e^{\frac{i\Dlt t}\hb\cH_N(j\Dlt t)}e^{\frac{i(t-[t/\Dlt t]\Dlt t)}\hb\cH_N([\frac{t}{\Dlt t}]\Dlt t)}\,,
\\
&U(t,s):=U(s,t)^*\,.
\ea
$$
Henceforth we denote for simplicity
\be\lb{DefcU}
\cU(t,s)A:=U(t,s)AU(s,t)
\ee
for each $A\in\cL(\fH)$; hence
\be\lb{wtR=}
\wtilde R(t)=\cU(t,0)R^{in}\,,\qquad t\ge 0\,.
\ee

\smallskip
Since our purpose is to find an error estimate for the RBM that is independent of the particle number $N$, we first need to define in terms of $R(t)$ and $\wtilde R(t)$ {\it quantities of interest} to be compared that are {\it independent
of $N$}. For instance one cannot hope to use the trace-norm of $\wtilde R(t)-R(t)$ since both $\wtilde R(t)$ and $R(t)$, and the trace-norm itself for elements of $\cL^1(\fH_N)$ significantly depend on $N$. (There are other 
reasons for not using the trace-norm in this context, which will be explained later.) A common practice when considering large systems of identical particles is to study the reduced density operators. Assume that $R^{in}$ has
an integral kernel $r^{in}\equiv r^{in}(x_1,\ldots,x_N;y_1,\ldots,y_N)$ satisfying the symmetry
\be\lb{RinSym}
r^{in}(x_1,\ldots,x_N;y_1,\ldots,y_N)=r^{in}(x_{\si(1)},\ldots,x_{\si(N)};y_{\si(1)},\ldots,y_{\si(N)})
\ee
for a.e. $(x_1,\ldots,x_N;y_1,\ldots,y_N)\in\bR^{2dN}$ and each permutation $\si\in\fS_N$. Then, for each $t\ge 0$, the $N$-body density operator $R(t)$ solution of \eqref{NBodyvN} satisfies the same symmetry, i.e. it has an 
integral kernel of the form $r(t;x_1,\ldots,x_N;y_1,\ldots,y_N)$ such that
\be\lb{RtSym}
r(t;x_1,\ldots,x_N;y_1,\ldots,y_N)=r(t;x_{\si(1)},\ldots,x_{\si(N)};y_{\si(1)},\ldots,y_{\si(N)})
\ee
for a.e. $(x_1,\ldots,x_N;y_1,\ldots,y_N)\in\bR^{2dN}$, all $t\ge 0$ and each permutation $\si\in\fS_N$. The $1$-particle reduced density operator of $R(t)\in\cD(\fH_N)$ is $R_\indc(t)\in\cD(\fH)$ defined by the integral kernel
\be\lb{R1t}
r_\indc(t,x,y):=\int_{(\bR^d)^{N-1}}r(t;x,z_2,\ldots,z_N;y,z_2,\ldots,z_N)dz_2\ldots dz_N\,.
\ee
(This operation is legitimate for a trace-class operator $R$ on $\fH_N$: indeed, $R$ has an integral kernel $r(x_1,\ldots,x_N;y_1,\ldots,y_N)$ such that
$$
(z_1,\ldots,z_N)\mapsto r(x_1+z_1,\ldots,x_N+z_N;x_1,\ldots,x_N)
$$
belongs to $C(\bR^{dN}_{z_1,\ldots,z_N};L^1(\bR^{dN}_{x_1,\ldots,x_N}))$ according to Footnote 1 on p. 61 in \cite{FGTPaulARMA2017}.) 

Even if $R^{in}$ satisfies the symmetry \eqref{RinSym}, in general $\wtilde R(t)$ does not satisfy the symmetry analogous to \eqref{RtSym} for $t>0$ (with $r$ replaced with $\wtilde r$, an integral kernel for $\wtilde R(t)$) 
because the random batch potential
$$
\sum_{1\le l<n\le N}\bT_t(l,n)V(x_l-x_n)
$$
is not invariant under permutations of the particle labels, at variance with the $N$-body potential
$$
\frac1{N-1}\sum_{1\le l<n\le N}V(x_l-x_n)\,.
$$
For that reason, the $1$-particle reduced density operator of $\wtilde R(t)$ is $\wtilde R_\indc(t)\in\cD(\fH)$ defined for all $t>0$ by the integral kernel
\be\lb{wtR1t}
\wtilde r_\indc(t,x,y):=\frac1N\sum_{j=1}^N\int_{(\bR^d)^{N-1}}\wtilde r(t;Z_{j,N}[x],Z_{j,N}[y])d\hat Z_{j,N}\,,
\ee
with the notation
$$
Z_{j,N}[x]:=z_1,\ldots,z_{j-1},x,z_{j+1}\ldots,z_N\,,\quad d\hat Z_{j,N}=dz_1\ldots dz_{j-1}dz_{j+1}\ldots dz_N\,.
$$
(Obviously \eqref{wtR1t} holds with $r_\indc$ and $r$ in the place of $\wtilde r_\indc$ and $\wtilde r$ respectively because of the symmetry \eqref{RtSym}.)

\smallskip
Our main result on the convergence of the RBM for the $N$-body von Neumann equation \eqref{NBodyvN} (or for the $N$-body Schr\"odinger equation \eqref{NBodySchro}) is stated in terms of the Wigner functions of the 
density operators $R(t)$ and $\wtilde R(t)$. We first recall the definition of the Wigner function of an operator $S\in\cL^2(\fH)$: let $s\equiv s(x,y)$ be an integral kernel for $S$. Then $s\in L^2(\bR^d\times\bR^d)$ and the 
Wigner function of $S$ is the element of $L^2(\bR^d\times\bR^d)$ defined by the formula
$$
W_\hb[S](x,\cdot):=\tfrac1{(2\pi)^d}\cF\big(y\mapsto s(x+\tfrac12\hb y,x-\tfrac12\hb y)\big)\quad\text{ for a.e. }x\in\bR^d\,,
$$
where $\cF$ designates the Fourier transform on $L^2(\bR^d)$. If the argument of $\cF$ is integrable in $y$, then
$$
W_\hb[S](x,\xi):=\tfrac1{(2\pi)^d}\int_{\bR^d}s(x+\tfrac12\hb y,x-\tfrac12\hb y)e^{-i\xi\cdot y}dy\,.
$$
When $S=|\psi\ra\la\psi$ with $\psi\in\fH$, the Wigner function of $S$ is often denoted $W_\hb[\psi]$. The reader is referred to \cite{LionsPaul} for more details on the Wigner function.

For each integer $M\ge 1$, we also introduce the dual norm
$$
|||f|||_{-M}\!:=\!\sup\left\{\left|\iint_{\bR^d\times\bR^d}f(x,\xi)\overline{a(x,\xi)}dxd\xi\right|\quad\left|\ba{}&\,\,\,\,\,a\in C_c(\bR^d\times\bR^d)\,,\,\,\text{ and }
\\
&\max_{|\a|,|\b|\le M\atop |\a|+|\b|>0}\|\d_x^\a\d_\xi^\b a\|_{L^\infty(\bR^d\times\bR^d)}\!\le\! 1\ea\right.\right\}\,.
$$

\begin{Thm}\lb{T-Main}
Assume that $N\ge 2$ and that $V\in C(\bR^d)$ is a real-valued function such that
$$
V(z)=V(-z)\text{ for all }z\in\bR^d\,,\quad\lim_{|z|\to+\infty}V(z)=0\,,\quad\text{ and }\int_{\bR^d}(1+|\om|^2)|\hat V(\om)|d\om<\infty\,.
$$
Let $R^{in}\in\cD(\fH_N)$, and let $R(t)$ and $\wtilde R(t)$ be defined respectively by \eqref{R=} and \eqref{wtR=}. Let $R_\indc(t)$ and $\wtilde R_\indc(t)$ be the single-particle reduced density operators defined in terms
of $R(t)$ and $\wtilde R(t)$ by \eqref{R1t}.

Then there exists a constant $\g_d>0$ depending only on the dimension $d$ of the configuration space such that, for each $t>0$, one has
\be\lb{ErrWig}
\ba
|||W_\hb[\bE\wtilde R_\indc(t)]-W_\hb[R_\indc(t)]|||_{-[d/2]-3}
\\
\le 2\g_d\Dlt te^{6t\max(1,\sqrt{d}L(V))}\L(V)(2+3t\L(V)\max(1,\Dlt t)+4\sqrt{d}L(V)t\Dlt t)&\,.
\ea
\ee
where $\bE$ is the mathematical expectation and $\Dlt t$ the reshuffling time-step in the definition of the random batch Hamiltonian \eqref{RBHam}, while
$$
L(V):=\tfrac1{(2\pi)^d}\int_{\bR^d}|\om|^2|\hat V(\om)|d\om\,,\qquad\L(V):=\tfrac1{(2\pi)^d}\int_{\bR^d}\sum_{\mu=1}^d|\om^\mu||\hat V(\om)|d\om\,,
$$
where $\omega^\nu$ is the $\nu$-th component of $\omega$.
\end{Thm}

\smallskip
Notice that the above result holds for the most general $N$-particle initial density operator $R^{in}$.

This error estimate satisfies both conditions (a) and (b). That it satisfies (a) is obvious, since $N$ appears on neither side of \eqref{ErrWig}. That it satisfies (b) is seen with the help of Theorem III.1 in \cite{LionsPaul}.
Indeed, for each $t\ge 0$, the operators $R(t)$ and $\wtilde R(t)$ define two bounded families of elements of $\cD(\fH)$ indexed by $N\ge 2$ and $\hb\in(0,1)$. Thus, $W_\hb[\bE\wtilde R_\indc(t)]$ and $W_\hb[R_\indc(t)]$
are relatively compact in the dual space $\cA'$ defined in Proposition III.1 of \cite{LionsPaul}, and the limit points of these families as $\hb\to 0$ are positive measures $\wtilde\mu(t)$ and $\mu(t)$ on the phase space 
$\bR^d\times\bR^d$. With relatively mild tightness assumptions on the behavior of $W_\hb[\bE\wtilde R_\indc(t)]$ and $W_\hb[R_\indc(t)]$ in the limit as $|x|+|\xi|\to\infty$, these ``Wigner measures'' $\wtilde\mu(t)$ and $\mu(t)$ 
encode the behavior of the reduced density operators $\wtilde R_\indc(t)$ and $R_\indc(t)$ in the semiclassical regime. After checking how the dual norm $|||\cdot|||_{-[d/2]-2}$ behaves on weakly-* converging sequences in 
$\cA'$, one can therefore hope that \eqref{ErrWig} implies an estimate for the difference of the $N$-body reduced Wigner measure $\mu(t)$, and the expected value of its random batch analogue $\wtilde\mu(t)$, since the right 
hand side of \eqref{ErrWig} does not involve $\hb$. 

The error estimate \eqref{ErrWig} can also be stated directly in terms of the density operators $R_\indc(t)$ and $\bE\wtilde R_\indc(t)$: see formula \eqref{LastIneq2} below. This formulation of the error bound involves a new 
metric $d_\hb$ on the set of density operators on $\fH$, introduced in Definition \ref{D-Defdhb}.

%%%%%%%%%%%%%%%%%%%%%%%%%%%%%%%%%%%%%%%%%%%%%%%%%%%%%%%%%%%%%%%%%%%%%%%%%%%%%%%%%%%%%%%%%%%%%%%%%%%%%%%%%%%%%

\section{Proof of Theorem \ref{T-Main}}\lb{S-Main}

%%%%%%%%%%%%%%%%%%%%%%%%%%%%%%%%%%%%%%%%%%%%%%%%%%%%%%%%%%%%%%%%%%%%%%%%%%%%%%%%%%%%%%%%%%%%%%%%%%%%%%%%%%%%%

The proof of Theorem \ref{T-Main} makes critical use of rather different key ingredients (such as the mutual independence of the reshuffling permutations $\si_j$, semiclassical estimates on the interaction terms, together
with a careful choice of test operators in the weak formulations of the $N$-body and random batch dynamics, and a quantitative version of the Calderon-Vaillancourt theorem), and as a result, is rather involved. We shall
therefore decompose our argument in seven steps. Each step addresses one of the key issues in the error estimate obtained in Theorem \ref{T-Main}.

\subsection{Using the weak formulations of \eqref{NBodySchro} and \eqref{RBSchro}}

%%%%%%%%%%%%%%%%%%%%%%%%%%%%%%%%%%%%%%%%%%%%%%%%%%%%%%%%%%%%%%%%%%%%%%%%%%%%%%%%%%%%%%%%%%%%%%%%%%%%%%%%%%%%%

For each $A\in\cL(\fH)$ and all $k=2,\ldots,N-1$, we set
$$
J_1A:=A\otimes I_{\fH}^{\otimes(N-1)}\,,\quad J_kA:=I_{\fH}^{\otimes(k-1)}\otimes A\otimes I_{\fH}^{\otimes(N-k)}\,,\quad J_NA=I_{\fH}^{\otimes(N-1)}\otimes A\,.
$$
With $A\in\cL(\fH)$ to be specified later, let
$$
B_N(s):=\cU(s,t)\frac1N\sum_{k=1}^NJ_kA\,.
$$

By the Duhamel formula,
$$
\ba
B_N(t)=e^{-it\cH_N/\hb}B_N(0)e^{+it\cH_N/\hb}
\\
+\frac1{i\hb}\int_0^te^{-i(t-s)\cH_N/\hb}\left[\sum_{1\le m<n\le N}\left(\bT_s(m,n)-\tfrac1{N-1}\right)V_{mn},B_N(s)\right]e^{+i(t-s)\cH_N/\hb}ds&\,,
\ea
$$
with the notation
$$
V_{mn}:=\text{ multiplication by }V(x_m-x_n)\,.
$$

Because of \eqref{R=} and \eqref{wtR=}, one has
$$
\ba
\Tr_{\fH_N}(\wtilde R(t)B_N(t))=&\Tr(U(t,0)\wtilde R(0)U(t,0)^*U(t,0)B_N(0)U(t,0)^*)
\\
=&\Tr_{\fH_N}(\wtilde R(0)B_N(0))
\ea
$$
by cyclicity of the trace, and because $U(t,s)$ is unitary on $\fH_N$. On the other hand
$$
\ba
\Tr_{\fH_N}(R(t)B_N(t))=&\Tr(e^{-it\cH_N/\hb}R^{in}e^{+it\cH_N/\hb}B_N(t))
\\
=&\Tr_{\fH_N}(R^{in}e^{+it\cH_N/\hb}B_N(t)e^{-it\cH_N/\hb})
\ea
$$
so that
$$
\ba
\Tr_{\fH_N}(R(t)B_N(t))-\Tr(R^{in}B_N(0))
\\
=\frac1{i\hb}\int_0^t\Tr\left(R^{in}e^{\frac{is}{\hb}\cH_N}\left[\sum_{1\le m<n\le N}\left(\bT_s(m,n)-\tfrac1{N-1}\right)V_{mn},B_N(s)\right]e^{-\frac{is}{\hb}\cH_N}\right)ds
\\
=\frac1{i\hb}\int_0^t\Tr\left(R(s)\left[\sum_{1\le m<n\le N}\left(\bT_s(m,n)-\tfrac1{N-1}\right)V_{mn},B_N(s)\right]\right)ds
\ea
$$
again by cyclicity of the trace. Therefore
$$
\ba
\Tr_{\fH_N}((\wtilde R(t)-R(t))B_N(t))
\\
=-\frac1{i\hb}\int_0^t\Tr\left(R(s)\left[\sum_{1\le m<n\le N}\left(\bT_s(m,n)-\tfrac1{N-1}\right)V_{mn},B_N(s)\right]\right)ds&\,.
\ea
$$

With our choice of $B_N(t)$, this last identity is recast as
$$
\ba
\Tr_{\fH_N}\left((\wtilde R(t)-R(t))\frac1N\sum_{k=1}^NJ_kA\right)
\\
=\frac{i}{\hb}\sum_{j=1}^{[\frac{t}{\Dlt t}]}\int_{(j-1)\Dlt t}^{j\Dlt t}\Tr\left(\left[\sum_{1\le m<n\le N}\left(\bT_s(m,n)-\tfrac1{N-1}\right)V_{mn},B_N(s)\right]R(s)\right)ds
\\
+\frac{i}{\hb}\int_{[\frac{t}{\Dlt t}]\Dlt t}^t\Tr\left(\left[\sum_{1\le m<n\le N}\left(\bT_s(m,n)-\tfrac1{N-1}\right)V_{mn},B_N(s)\right]R(s)\right)ds&\,.
\ea
$$
Hence, taking the expectation over random reshufflings, and using the definitions \eqref{R1t} and \eqref{wtR1t} of reduced density operators, we arrive at the identity
$$
\ba
\Tr_{\fH}\left((\bE\wtilde R_\indc(t)-R_\indc(t))A\right)=\Tr_{\fH_N}\left((\bE\wtilde R(t)-R(t))\frac1N\sum_{k=1}^NJ_kA\right)
\\
=\frac{i}{\hb}\sum_{j=1}^{[\frac{t}{\Dlt t}]}\int_{(j-1)\Dlt t}^{j\Dlt t}\Tr\left(R(s)\bE\left[\sum_{1\le m<n\le N}\left(\bT_s(m,n)-\tfrac1{N-1}\right)V_{mn},B_N(s)\right]\right)ds
\\
+\frac{i}{\hb}\int_{[\frac{t}{\Dlt t}]\Dlt t}^t\Tr\left(R(s)\bE\left[\sum_{1\le m<n\le N}\left(\bT_s(m,n)-\tfrac1{N-1}\right)V_{mn},B_N(s)\right]\right)ds&\,.
\ea
$$

\subsection{Using the independence of $\si_1,\si_2,\ldots$}

%%%%%%%%%%%%%%%%%%%%%%%%%%%%%%%%%%%%%%%%%%%%%%%%%%%%%%%%%%%%%%%%%%%%%%%%%%%%%%%%%%%%%%%%%%%%%%%%%%%%%%%%%%%%%

For all $S\in\cL(\fH_N)$, denote
$$
\cU_0(t)S:=\left(e^{it\hb\Dlt/2}\right)^{\otimes N}S\left(e^{-it\hb\Dlt/2}\right)^{\otimes N}\,.
$$
Observe that
$$
\ba
\Tr_{\fH_N}\left(R(s)\bE\left[\sum_{1\le m<n\le N}\left(\bT_s(m,n)-\tfrac1{N-1}\right)V_{mn},B_N(s)\right]\right)
\\
=\Tr_{\fH_N}\left(R(s)\bE\left[\sum_{1\le m<n\le N}\left(\bT_s(m,n)-\tfrac1{N-1}\right)V_{mn},\Dlt B_N(s,j))\right]\right)&\,,
\ea
$$
with 
$$
\Dlt B_N(s,j):=B_N(s)-\cU_0(s-j\Dlt t)B_N(j\Dlt t)\,,
$$
since
$$
\ba
\Tr_{\fH_N}\left(R(s)\bE\left[\sum_{1\le m<n\le N}\left(\bT_s(m,n)-\tfrac1{N-1}\right)V_{mn},\cU_0(s-j\Dlt t)B_N(j\Dlt t)\right]\right)
\\
=\Tr_{\fH_N}\left(R(s)\left[\sum_{1\le m<n\le N}\bE\left(\bT_s(m,n)-\tfrac1{N-1}\right)V_{mn},\cU_0(s-j\Dlt t)\bE B_N(j\Dlt t)\right]\right)&=0\,.
\ea
$$
The penultimate equality follows from the independence of the $\si_j$'s, since 
$$
B_N(j\Dlt t)=B_N(j\Dlt t+0)
$$
involves only $\si_{j+1},\ldots,\si_{[t/\Dlt t]+1}$, while $\cU_0(s-j\Dlt t)$ is deterministic and $\bT_s(m,n)$ only depends on $\si_j$. As for the last equality, it comes from the identity
$$
\bE\bT_s(m,n)=\frac1{N-1}\,,\quad\text{ for all }1\le m<n\le N\text{ and }s\ge 0\,.
$$
For this last identity, see the proof of Lemma 3.1 in \cite{SJinLLiJGLiu}, and especially the second formula after (3.13) on p. 8 in \cite{SJinLLiJGLiu}.

Therefore
\be\lb{WeakIdent}
\ba
\Tr_{\fH}\left((\bE\wtilde R_\indc(t)-R_\indc(t))A\right)=\Tr_{\fH_N}\left((\bE\wtilde R(t)-R(t))\frac1N\sum_{k=1}^NJ_kA\right)
\\
=\frac{i}{\hb}\sum_{j=1}^{[\frac{t}{\Dlt t}]}\int_{(j-1)\Dlt t}^{j\Dlt t}\Tr_{\fH_N}\Bigg(R(s)\bE\Bigg[\sum_{1\le m<n\le N}\left(\bT_s(m,n)-\tfrac1{N-1}\right)V_{mn},\Dlt B_N(s,j)\Bigg]\Bigg)ds
\\
+\frac{i}{\hb}\int_{[\frac{t}{\Dlt t}]\Dlt t}^t\Tr_{\fH_N}\left(R(s)\bE\left[\sum_{1\le m<n\le N}\left(\bT_s(m,n)-\tfrac1{N-1}\right)V_{mn},B_N(s)\right]\right)ds&\,.
\ea
\ee

\subsection{Semiclassical potential estimate}\lb{SS-SCPotEst}

%%%%%%%%%%%%%%%%%%%%%%%%%%%%%%%%%%%%%%%%%%%%%%%%%%%%%%%%%%%%%%%%%%%%%%%%%%%%%%%%%%%%%%%%%%%%%%%%%%%%%%%%%%%%%

The previous formula makes it obvious that our error analysis requires estimating commutators of various operators with the interaction potential. Besides, all these terms involve a $1/\hb$ prefactor. With a view
towards obtaining uniform as $\hb\to 0$ error estimates, one should avoid by all means using bounds of the type
$$
\left\|\frac1\hb[V_{mn},S]\right\|\le\frac2\hb\|V\|_{L^\infty(\bR^d)}\|S\|\,.
$$
We shall use instead the following lemma (see \cite{FGTPaulEmpir} on p. 1048).

\begin{Lem}\lb{L-[fT]}
Let $f\equiv f(x)$ be an element of $C^1_0(\bR^d;\bC)$ such that $\hat f$ and $\widehat\grad f$ belong to $L^1(\bR^d)$. Then, for each $T\in\cL(\fH)$, one has
$$
\|[f,T]\|\le\L(f)\max_{1\le\nu\le d}\|[x^\nu,T]\|\,,
$$
with
$$
\L(f):=\tfrac1{(2\pi)^d}\int_{\bR^d}\sum_{\mu=1}^d|\om^\mu||\hat f(\om)|d\om\,.
$$
\end{Lem}

\begin{proof}
Let $E_\om\in\cL(\fH)$ be the operator defined by $E_\om\psi(x)=e^{i\om\cdot x}\psi(x)$ for all $x\in\bR^d$. Then
$$
[f,T]=\tfrac1{(2\pi)^d}\int_{\bR^d}\hat f(\om)[E_\om,T]d\om
$$
and
$$
[E_\om,T]E_\om^*=\int_0^1\frac{d}{dt}(E_{t\om}TE_{t\om}^*)dt=\int_0^1E_{t\om}[i\om\cdot x,T]E_{t\om}^*dt
$$
so that
$$
\|[E_\om,T]\|\le\max_{1\le\nu\le d}\|[x^\nu,T]\|\sum_{\mu=1}^d|\om^\mu|\,.
$$
Hence
$$
\|[f,T]\|\le\max_{1\le\nu\le d}\|[x^\nu,T]\tfrac1{(2\pi)^d}\int_{\bR^d}|\hat f(\om)|\sum_{\mu=1}^d|\om^\mu|d\om\,,
$$
which implies the desired bound.
\end{proof}

\smallskip
We use this lemma to control the terms $[V_{mn},B_N(s)]$ and $[V_{mn},\Dlt B_N(s,j)]$ for each $m,n=1,\ldots,N$. First, one has
\be\lb{[VB]<}
\|[V_{mn},B_N(s)]\|\le\L(V)(\max_{1\le\mu\le d}\|[x_m^\mu,B_N(s)]\|+\max_{1\le\nu\le d}\|[x_n^\nu,B_N(s)]\|)\,,
\ee
and
\be\lb{[VDB]<}
\|[V_{mn},\Dlt B_N(s,j)]\|\le\L(V)(\max_{1\le\mu\le d}\|[x_m^\mu,\Dlt B_N(s,j)]\|+\max_{1\le\nu\le d}\|[x_n^\nu,\Dlt B_N(s,j)]\|)\,.
\ee
Hence
\be\lb{Last[]}
\ba
\left\|\sum_{1\le m<n\le N}\left[\left(\bT_s(m,n)-\tfrac1{N-1}\right)V_{mn},B_N(s)\right]\right\|\le 2\L(V)\sum_{m=1}^N\max_{1\le\mu\le d}\|[x_m^\mu,B_N(s)]\|&\,,
\ea
\ee
and
\be\lb{Gen[]}
\ba
\left\|\sum_{1\le m<n\le N}\left[\left(\bT_s(m,n)-\tfrac1{N-1}\right)V_{mn},\Dlt B_N(s,j)\right]\right\|
\\
\le 2\L(V)\sum_{m=1}^N\max_{1\le\mu\le d}\|[x_m^\mu,\Dlt B_N(s,j)]\|&\,.
\ea
\ee

For each $m=1,\ldots,N$, we henceforth denote by $\wtilde m(t)$ the unique index in $\{1,\ldots,N\}$ different from $m$ and in the same batch as $m$ at time $t$. In other words, $\tilde m(t)$ is defined by 
the  following two conditions:
\be\lb{Deftm}
\wtilde m(t)\not=m\quad\text{ and }\quad\bT_t(m,\wtilde m(t))=1\,.
\ee
By the Duhamel formula
\be\lb{FlaDeltaB}
\Dlt B_N(s,j)=\frac1{i\hb}\int_{j\Dlt t}^s\tfrac12\sum_{l=1}^N\cU_0(s-\tau)[V_{l,{\wtilde l}(s)},B(\tau)]d\tau\,,
\ee
so that, for each $m=1,\ldots,N$ and each $\mu=1,\ldots,d$, one has
$$
\|[x_m^\mu,\Dlt B_N(s,j)]\|\le\frac1{\hb}\int_s^{j\Dlt t}\tfrac12\sum_{l=1}^N\|[x_m^\mu,\cU_0(s-\tau)[V_{l,{\wtilde l}(s)},B(\tau)]]\|d\tau\,.
$$

An elementary computation shows that
$$
[x_m^\mu,\cU_0(\th)R]=\cU_0(\th)([x_m^\mu,R]+\th[-i\hb\d_{x_m^\mu},R])\,;
$$
hence
$$
\|[x_m^\mu,\cU_0(\th)R]\|\le\|[x_m^\mu,R]\|+|\th|\|[-i\hb\d_{x_m^\mu},R]\|\,.
$$

Therefore
$$
\ba
\|[x_m^\mu,\Dlt B_N(s,j)]\|
\\
\le\frac1{2\hb}\int_s^{j\Dlt t}\sum_{l=1}^N(\|[x_m^\mu,[V_{l,{\wtilde l}(s)},B(\tau)]]\|+(\tau-s)\|[-i\hb\d_{x_m^\mu},[V_{l,{\wtilde l}(s)},B(\tau)]]\|)d\tau
\\
\le\frac1{2\hb}\int_s^{j\Dlt t}\sum_{l=1}^N(\|[V_{l,{\wtilde l}(s)},[x_m^\mu,B(\tau)]]\|+(\tau-s)\|[V_{l,{\wtilde l}(s)},[-i\hb\d_{x_m^\mu},B(\tau)]]\|)d\tau
\\
+\frac1{\hb}\int_s^{j\Dlt t}(\tau-s)\|[-i\hb\d_\mu V_{m,{\wtilde m}(s)},B(\tau)]\|d\tau&\,.
\ea
$$
In other words,
\be\lb{[xmmuDltB]}
\ba
\sum_{m=1}^N\max_{1\le\mu\le d}\|[x_m^\mu,\Dlt B_N(s,j)]\|
\\
\le\frac{\L(V)}{\hb}\int_s^{j\Dlt t}\sum_{1\le l,m\le N}\max_{1\le\l,\mu\le d}\|[x_l^\l,[x_m^\mu,B(\tau)]]\|d\tau
\\
+\frac{\L(V)\Dlt t}{\hb}\int_s^{j\Dlt t}\sum_{1\le l,m\le N}\max_{1\le\l,\mu\le d}\|[x_l^\l,[-i\hb\d_{x_m^\mu},B(\tau)]]\|d\tau
\\
+2\sqrt{d}L(V)\Dlt t\int_s^{j\Dlt t}\sum_{m=1}^N\max_{1\le\mu\le d}\|[x_m^\mu,B(\tau)]\|d\tau&\,.
\ea
\ee

\subsection{First order semiclassical estimates}\lb{SS-1stSC}

%%%%%%%%%%%%%%%%%%%%%%%%%%%%%%%%%%%%%%%%%%%%%%%%%%%%%%%%%%%%%%%%%%%%%%%%%%%%%%%%%%%%%%%%%%%%%%%%%%%%%%%%%%%%%

In view of the inequality above, a key issue is therefore to bound operators of the form 
$$
[x_m^\mu,B(\tau)]\,,\qquad[-i\hb\d_{x_m^\mu},B(\tau)]\,,
$$
and
$$
[x_l^\l,[x_m^\mu,B(\tau)]]\,,\quad[x_l^\l,[-i\hb\d_{x_m^\mu},B(\tau)]]\,,\quad[-i\hb\d_{x_l^\l},[x_m^\mu,B(\tau)]]\,.
$$
This is done in the present section and the next. The bounds on the first two quantities above follow the proof of Lemma 4.1 in \cite{FGTPaulEmpir}, which is itself based on the earlier analysis in Appendix B of 
\cite{BeneJakPortaSaffSchlein}, or Appendix C of \cite{BenePortaSaffSchlein}.

We recall that
$$
i\hbar\d_sB_N(s)=[\cH_N(s),B_N(s)]\,,\qquad B_N(t)=\frac1N\sum_{m=1}^NJ_mA\,.
$$
Hence
$$
i\hbar\d_s[x_m^\mu,B_N(s)]=[\cH_N(s),[x_m^\mu,B_N(s)]]+[[x_m^\mu,\cH_N(s)],B_N(s)]\,,
$$
and 
$$
[x_m^\mu,\cH_N(s)]=[x_m^\mu,-\tfrac12\hbar^2\Dlt_{x_m}]=i\hb(-i\hb\d_{x_m^\mu})\,,
$$
so that
$$
i\hbar\d_s[x_m^\mu,B_N(s)]=[\cH_N(s),[x_m^\mu,B_N(s)]]+i\hb[-i\hb\d_{x_m^\mu},B_N(s)]\,.
$$
Thus
$$
\|[x_m^\mu,B_N(s)]\|\le\|[x_m^\mu,B_N(t)]\|+\int_s^t\|[-i\hb\d_{x_m^\mu},B_N(\tau)]\|d\tau\,.
$$
Likewise
$$
i\hbar\d_s[-i\hb\d_{x_m^\mu},B_N(s)]=[\cH_N(s),[-i\hb\d_{x_m^\mu},B_N(s)]]+[[-i\hb\d_{x_m^\mu},\cH_N(s)],B_N(s)]\,,
$$
and 
$$
[-i\hb\d_{x_m^\mu},\cH_N(s)]=\tfrac12\sum_{l=1}^N[-i\hb\d_{x_m^\mu},V(x_l-x_{\wtilde l(s)})]=-i\hb\d_\mu V(x_m-x_{\wtilde m(s)})\,,
$$
so that
$$
i\hbar\d_s[-i\hb\d_{x_m^\mu},B_N(s)]=[\cH_N(s),[-i\hb\d_{x_m^\mu},B_N(s)]]-i\hb[\d_\mu V(x_m-x_{\wtilde m(s)}),B_N(s)]\,,
$$
Therefore
$$
\ba
{}\|[-i\hb\d_{x_m^\mu},B_N(s)]\|\le&\|[-i\hb\d_{x_m^\mu},B_N(t)]\|
\\
&+\sqrt{d}L(V)\int_s^t(\max_{1\le\ka\le d}\|[x_m^\ka,B_N(\tau)]\|+\max_{1\le\ka\le d}\|[x_{\wtilde m(\tau)}^\ka,B_N(\tau)])d\tau
\ea
$$

Set
\be\lb{DefM1}
M_1(s):=\sum_{m=1}^N\max_{1\le\mu\le d}\|[x_m^\mu,B_N(s)]\|+\sum_{m=1}^N\max_{1\le\mu\le d}\|[-i\hb\d_{x_m^\mu},B_N(s)]\|\,.
\ee
Then, one has
$$
M_1(s)\le M_1(t)+\max(1,2\sqrt{d}L(V))\int_s^tM_1(\tau)
$$
so that, by Gronwall's lemma
\be\lb{M1s<}
M_1(s)\le M_1(t)e^{(t-s)\max(1,2\sqrt{d}L(V))}\,.
\ee

Finally, since 
$$
B_N(t)=\frac1N\sum_{k=1}^NJ_kA\,,
$$
one has
$$
[x_m^\mu,B_N(t)]=\frac1NJ_m[x^\mu,A]\,,\qquad[-i\hb\d_{x_m^\mu},B_N(t)]=\frac1NJ_m[-i\hb\d_\mu,A]\,.
$$
Hence
$$
\|[x_m^\mu,B_N(t)]\|\le\frac1N\|[x^\mu,A]\|\,,\qquad\|[-i\hb\d_{x_m^\mu},B_N(t)]\|\le\frac1N\|[-i\hb\d_\mu,A]\|\,,
$$
so that
\be\lb{M1t}
M_1(t)\le\max_{1\le\mu\le d}\|[x^\mu,A]\|+\max_{1\le\mu\le d}\|[-i\hb\d_\mu,A]\|
\ee
and therefore
$$
M_1(s)\le(\max_{1\le\mu\le d}\|[x^\mu,A]\|+\max_{1\le\mu\le d}\|[-i\hb\d_\mu,A]\|)e^{(t-s)\max(1,2\sqrt{d}L(V))}\,.
$$

\subsection{Second Order Semiclassical Estimates}\lb{SS-2ndSC}

%%%%%%%%%%%%%%%%%%%%%%%%%%%%%%%%%%%%%%%%%%%%%%%%%%%%%%%%%%%%%%%%%%%%%%%%%%%%%%%%%%%%%%%%%%%%%%%%%%%%%%%%%%%%%

One has
$$
\ba
i\hbar\d_s[x_n^\nu,[x_m^\mu,B_N(s)]]=&[\cH_N(s),[x_n^\nu,[x_m^\mu,B_N(s)]]]
\\
&+[[x_n^\nu,\cH_N(s)],[x_m^\mu,B_N(s)]]+i\hb[x_n^\nu,[-i\hb\d_{x_m^\mu},B_N(s)]]
\\
=&[\cH_N(s),[x_n^\nu,[x_m^\mu,B_N(s)]]]
\\
&+i\hb[-i\hb\d_{x_n^\nu},[x_m^\mu,B_N(s)]]+i\hb[x_n^\nu,[-i\hb\d_{x_m^\mu},B_N(s)]]\,.
\ea
$$
Notice that
$$
[-i\hb\d_{x_n^\nu},[x_m^\mu,B_N(s)]]=[x_m^\mu,[-i\hb\d_{x_n^\nu},B_N(s)]]\,.
$$

Likewise
$$
\ba
i\hbar\d_s[-i\hb\d_{x_n^\nu},[x_m^\mu,B_N(s)]]=&[\cH_N(s),[-i\hb\d_{x_n^\nu},[x_m^\mu,B_N(s)]]]
\\
&+[[-i\hb\d_{x_n^\nu},\cH_N(s)],[x_m^\mu,B_N(s)]]
\\
&+i\hb[-i\hb\d_{x_n^\nu},[-i\hb\d_{x_m^\mu},B_N(s)]]
\\
=&[\cH_N(s),[-i\hb\d_{x_n^\nu},[x_m^\mu,B_N(s)]]]
\\
&-i\hb[\d_\nu V(x_n-x_{\wtilde n(s)}),[x_m^\mu,B_N(s)]]
\\
&+i\hb[-i\hb\d_{x_n^\nu},[-i\hb\d_{x_m^\mu},B_N(s)]]\,,
\ea
$$
and
$$
\ba
i\hbar\d_s[-i\hb\d_{x_n^\nu},[-i\hb\d_{x_m^\mu},B_N(s)]]=&[\cH_N(s),[-i\hb\d_{x_n^\nu},[-i\hb\d_{x_m^\mu},B_N(s)]]]
\\
&+[[-i\hb\d_{x_n^\nu},\cH_N(s)],[-i\hb\d_{x_m^\mu},B_N(s)]]
\\
&+i\hb[-i\hb\d_{x_n^\nu},[-i\hb\d_{x_m^\mu},B_N(s)]]
\\
=&[\cH_N(s),[-i\hb\d_{x_n^\nu},[-i\hb\d_{x_n^\nu},B_N(s)]]]
\\
&-i\hb[\d_\nu V(x_n-x_{\wtilde n(s)}),[-i\hb\d_{x_m^\mu},B_N(s)]]
\\
&+i\hb[-i\hb\d_{x_n^\nu},[\d_\mu V(x_m-x_{\wtilde m(s)}),B_N(s)]]\,.
\ea
$$
Since
$$
\ba
{}[-i\hb\d_{x_n^\nu},[\d_\mu V(x_m-x_{\wtilde m(s)}),B_N(s)]]=[\d_\mu V(x_m-x_{\wtilde m(s)}),[-i\hb\d_{x_n^\nu},B_N(s)]]
\\
+[[-i\hb\d_{x_n^\nu},\d_\mu V(x_m-x_{\wtilde m(s)})],B_N(s)]
\ea
$$
and
$$
\ba
{}[-i\hb\d_{x_n^\nu},\d_\mu V(x_m-x_{\wtilde m(s)})]=&-i\hb\d_\nu\d_\mu V(x_m-x_{\wtilde m(s)})(\de_{n,m}-\de_{n,\wtilde m(s)})
\\
=&-i\hb\d_\nu\d_\mu V(x_n-x_{\wtilde n(s)})+i\hb\d_\nu\d_\mu V(x_{\wtilde n(s)}-x_n)=0\,.
\ea
$$
Indeed, $V$ is even, so that $\d_\nu\d_\mu V$ is also even. Hence
$$
\ba
i\hbar\d_s[-i\hb\d_{x_n^\nu},[-i\hb\d_{x_m^\mu},B_N(s)]]=&[\cH_N(s),[-i\hb\d_{x_n^\nu},[-i\hb\d_{x_n^\nu},B_N(s)]]]
\\
&-i\hb[\d_\nu V(x_n-x_{\wtilde n(s)}),[-i\hb\d_{x_m^\mu},B_N(s)]]
\\
&+i\hb[\d_\mu V(x_m-x_{\wtilde m(s)}),[-i\hb\d_{x_n^\nu},B_N(s)]]\,.
\ea
$$

Using the Duhamel formula, one finds that
$$
\ba
\|[x_n^\nu,[x_m^\mu,B_N(s)]]\|\le\|[x_n^\nu,[x_m^\mu,B_N(s)]]\|+\int_s^t\|[x_m^\mu,[-i\hb\d_{x_n^\nu},B_N(\tau)]]\|d\tau
\\
+\int_s^t\|[x_n^\nu,[-i\hb\d_{x_m^\mu},B_N(\tau)]]\|d\tau
\ea
$$
and
$$
\ba
\|[-i\hb\d_{x_n^\nu},[x_m^\mu,B_N(s)]]\|\le&\|[-i\hb\d_{x_n^\nu},[x_m^\mu,B_N(t)]]\|
\\
&+\sqrt{d}L(V)\int_s^t\max_{1\le\ka\le d}\|[x_n^\ka,[x_m^\mu,B_N(\tau)]]\|d\tau
\\
&+\sqrt{d}L(V)\int_s^t\max_{1\le\ka\le d}\|[x_{\wtilde n(\tau)}^\ka,[x_m^\mu,B_N(\tau)]]\|d\tau
\\
&+\int_s^t\|[-i\hb\d_{x_n^\nu},[-i\hb\d_{x_m^\mu},B_N(\tau)]]\|d\tau\,,
\ea
$$
while
$$
\ba
\|[-i\hb\d_{x_n^\nu},[-i\hb\d_{x_m^\mu},B_N(s)]]\|\le&\|[-i\hb\d_{x_n^\nu},[-i\hb\d_{x_m^\mu},B_N(t)]]\|
\\
&+\sqrt{d}L(V)\int_s^t\max_{1\le\ka\le d}\|[x_n^\ka,[-i\hb\d_{x_m^\mu},B_N(\tau)]]\|d\tau
\\
&+\sqrt{d}L(V)\int_s^t\max_{1\le\ka\le d}\|[x_{\wtilde n(\tau)}^\ka,[-i\hb\d_{x_m^\mu},B_N(\tau)]]\|d\tau
\\
&+\sqrt{d}L(V)\int_s^t\max_{1\le\ka\le d}\|[x_m^\ka,[-i\hb\d_{x_n^\nu},B_N(\tau)]]\|d\tau
\\
&+\sqrt{d}L(V)\int_s^t\max_{1\le\ka\le d}\|[x_{\wtilde m(\tau)}^\ka,[-i\hb\d_{x_n^\nu},B_N(\tau)]]\|d\tau\,.
\ea
$$

Set
\be\lb{DefM2}
\ba
M_2(s):=&\sum_{1\le m,n\le N}\max_{1\le\mu,\nu\le d}\|[x_n^\nu,[x_m^\mu,B_N(s)]]\|
\\
&+\sum_{1\le m,n\le N}\max_{1\le\mu,\nu\le d}\|[-i\hb\d_{x_n^\nu},[x_m^\mu,B_N(s)]]\|
\\
&+\sum_{1\le m,n\le N}\max_{1\le\mu,\nu\le d}\|[-i\hb\d_{x_n^\nu},[-i\hb\d_{x_m^\mu},B_N(s)]]\|\,.
\ea
\ee
Therefore
$$
M_2(s)\le M_2(t)+6\max(1,\sqrt{d}L(V))\int_s^tM_2(\tau)d\tau\,.
$$

Since
$$
B_N(t)=\frac1N\sum_{k=1}^NJ_kA
$$
one has
\be\lb{RarefBt}
\ba
{}[x_n^\nu,[x_m^\mu,B_N(t)]]&=\frac{\de_{mn}}NJ_m[x^\nu,[x^\mu,A]]
\\
[-i\hb\d_{x_n^\nu},[x_m^\mu,B_N(t)]]&=\frac{\de_{mn}}NJ_m[-i\hb\d_\nu,[x^\mu,A]]
\\
[-i\hb\d_{x_n^\nu},[-i\hb\d_{x_m^\mu},B_N(t)]]&=\frac{\de_{mn}}NJ_m[-i\hb\d_\nu,[-i\hb\d_\mu,A]]\,.
\ea
\ee
Therefore
\be\lb{M2t}
\ba
M_2(t)\le&\max_{1\le\mu,\nu\le d}\|[x^\nu,[x^\mu,A]]\|+\max_{1\le\mu,\nu\le d}\|[-i\hb\d_\nu,[x^\mu,A]]\|
\\
&+\max_{1\le\mu,\nu\le d}\|[-i\hb\d_\nu,[-i\hb\d_\mu,A]]\|\,,
\ea
\ee
and
\be\lb{M2s<}
\ba
M_2(s)\le&\max_{1\le\mu,\nu\le d}\|[x^\nu,[x^\mu,A]]\|e^{6(t-s)\max(1,\sqrt{d}L(V))}
\\
&+\max_{1\le\mu,\nu\le d}\|[-i\hb\d_\nu,[x^\mu,A]]\|e^{6(t-s)\max(1,\sqrt{d}L(V))}
\\
&+\max_{1\le\mu,\nu\le d}\|[-i\hb\d_\nu,[-i\hb\d_\mu,A]]\|e^{6(t-s)\max(1,\sqrt{d}L(V))}\,.
\ea
\ee

\subsection{Implications of the semiclassical potential estimates}

%%%%%%%%%%%%%%%%%%%%%%%%%%%%%%%%%%%%%%%%%%%%%%%%%%%%%%%%%%%%%%%%%%%%%%%%%%%%%%%%%%%%%%%%%%%%%%%%%%%%%%%%%%%%%

At this point, we use \eqref{Gen[]}, \eqref{[xmmuDltB]} and \eqref{Last[]} together with the bounds obtained in the last two steps. 

First, we find that
$$
\ba
\frac1{\hb}\left|\int_{[\frac{t}{\Dlt t}]\Dlt t}^t\Tr\left(R_N(s)\bE\left[\sum_{1\le m<n\le N}\left(\bT_s(m,n)-\tfrac1{N-1}\right)V_{mn},B_N(s)\right]\right)ds\right|
\\
\le\frac1{\hb}\int_{[\frac{t}{\Dlt t}]\Dlt t}^t\bE\left\|\left[\sum_{1\le m<n\le N}\left(\bT_s(m,n)-\tfrac1{N-1}\right)V_{mn},B_N(s)\right]\right\|ds
\\
\le\frac1{\hb}\int_{[\frac{t}{\Dlt t}]\Dlt t}^t\tfrac1{N-1}\sum_{1\le m<n\le N}\left\|\left[V_{mn},B_N(s)\right]\right\|ds
\\
+\frac1{\hb}\int_{[\frac{t}{\Dlt t}]\Dlt t}^t\tfrac12\sum_{m=1}^N\left\|\left[V_{m,\wtilde m(s)},B_N(s)\right]\right\|ds&\,,
\ea
$$
so that
$$
\ba
\frac1{\hb}\left|\int_{[\frac{t}{\Dlt t}]\Dlt t}^t\Tr\left(R_N(s)\bE\left[\sum_{1\le m<n\le N}\left(\bT_s(m,n)-\tfrac1{N-1}\right)V_{mn},B_N(s)\right]\right)ds\right|
\\
\le\frac{\L(V)}{\hb}\int_{[\frac{t}{\Dlt t}]\Dlt t}^t\tfrac1{N-1}\sum_{1\le m<n\le N}(\max_{1\le\ka\le d}\left\|\left[x_m^\ka,B_N(s)\right]\right\|+\max_{1\le\ka\le d}\left\|\left[x_n^\ka,B_N(s)\right]\right\|)ds
\\
+\frac{\L(V)}{\hb}\int_{[\frac{t}{\Dlt t}]\Dlt t}^t\tfrac12\sum_{m=1}^N(\max_{1\le\ka\le d}\left\|\left[x_m^\ka,B_N(s)\right]\right\|+\max_{1\le\ka\le d}\left\|\left[x_{\wtilde m(s)}^\ka,B_N(s)\right]\right\|)ds
\\
\le\frac{2\L(V)}{\hb}\int_{[\frac{t}{\Dlt t}]\Dlt t}^t\sum_{m=1}^N\max_{1\le\ka\le d}\left\|\left[x_m^\ka,B_N(s)\right]\right\|ds\le\frac{2\L(V)}{\hb}\int_{[\frac{t}{\Dlt t}]\Dlt t}^tM_1(s)ds
\\
\le 2\L(V)\Dlt te^{t\max(1,2\sqrt{d}L(V))}\frac{M_1(t)}{\hb}&\,.
\ea
$$

Next
$$
\ba
\frac1{\hb}\left|\int_{(j-1)\Dlt t}^{j\Dlt t}\Tr\Bigg(R_N(s)\Bigg[\sum_{1\le m<n\le N}\left(\bT_s(m,n)-\tfrac1{N-1}\right)V_{mn},\Dlt B_N(s,j)\Bigg]\Bigg)ds\right|
\\
\le\frac1{\hb}\int_{(j-1)\Dlt t}^{j\Dlt t}\bE\left\|\Bigg[\sum_{1\le m<n\le N}\left(\bT_s(m,n)-\tfrac1{N-1}\right)V_{mn},\Dlt B_N(s,j)\Bigg]\right\|ds
\\
\le\frac1{\hb}\int_{(j-1)\Dlt t}^{j\Dlt t}\tfrac1{N-1}\sum_{1\le m<n\le N}\|[V_{mn},\Dlt B_N(s,j)]\|ds
\\
+\frac1{\hb}\int_{(j-1)\Dlt t}^{j\Dlt t}\tfrac12\sum_{m=1}^N\|[V_{m,\wtilde m(s)},\Dlt B_N(s,j)]\|ds&\,,
\ea
$$
so that
$$
\ba
\frac1{\hb}\left|\int_{(j-1)\Dlt t}^{j\Dlt t}\Tr\Bigg(R_N(s)\Bigg[\sum_{1\le m<n\le N}\left(\bT_s(m,n)-\tfrac1{N-1}\right)V_{mn},\Dlt B_N(s,j)\Bigg]\Bigg)ds\right|
\\
\le\frac{\L(V)}{\hb}\int_{(j-1)\Dlt t}^{j\Dlt t}\tfrac1{N-1}\sum_{1\le m<n\le N}(\max_{1\le\ka\le d}\left\|\left[x_m^\ka,\Dlt B_N(s,j)\right]\right\|+\max_{1\le\ka\le d}\left\|\left[x_n^\ka,\Dlt B_N(s,j)\right]\right\|)ds
\\
+\frac{\L(V)}{\hb}\int_{(j-1)\Dlt t}^{j\Dlt t}\tfrac12\sum_{m=1}^N(\max_{1\le\ka\le d}\left\|\left[x_m^\ka,\Dlt B_N(s,j)\right]\right\|+\max_{1\le\ka\le d}\left\|\left[x_{\wtilde m(s)}^\ka,\Dlt B_N(s,j)\right]\right\|)ds
\\
\le\frac{2\L(V)}{\hb}\int_{(j-1)\Dlt t}^{j\Dlt t}\sum_{m=1}^N\max_{1\le\mu\le d}\left\|\left[x_m^\ka,\Dlt B_N(s,j)\right]\right\|&\,.
\ea
$$
Recall that
$$
\ba
\sum_{m=1}^N\max_{1\le\mu\le d}\|[x_m^\mu,\Dlt B_N(s,j)]\|
\\
\le\frac{\L(V)}{\hb}\max(1,\Dlt t)\int_s^{j\Dlt t}M_2(\tau)d\tau+2\sqrt{d}L(V)\Dlt t\int_s^{j\Dlt t}M_1(\tau)d\tau
\\
\le\L(V)\Dlt t\max(1,\Dlt t)e^{6t\max(1,\sqrt{d}L(V))}\frac{M_2(t)}\hb
\\
+2\sqrt{d}L(V)\Dlt t^2e^{t\max(1,2\sqrt{d}L(V))}M_1(t)&\,.
\ea
$$
Therefore
$$
\ba
\frac1{\hb}\left|\int_{(j-1)\Dlt t}^{j\Dlt t}\Tr\Bigg(R_N(s)\Bigg[\sum_{1\le m<n\le N}\left(\bT_s(m,n)-\tfrac1{N-1}\right)V_{mn},\Dlt B_N(s,j)\Bigg]\Bigg)ds\right|
\\
\le 2\L(V)^2\Dlt t^2\max(1,\Dlt t)e^{6t\max(1,\sqrt{d}L(V))}\frac{M_2(t)}{\hb^2}
\\
+4\sqrt{d}\L(V) L(V)\Dlt t^3e^{t\max(1,2\sqrt{d}L(V))}\frac{M_1(t)}\hb&\,.
\ea
$$

Putting together all these elements of information leads to the bound
\be\lb{LastIneq}
\ba
\left|\Tr_\fH\left((\bE\wtilde R_\indc(t)-R_\indc(t))A\right)\right|=\left|\Tr_{\fH_N}\left((\bE\wtilde R(t)-R(t))\frac1N\sum_{k=1}^NJ_kA\right)\right|
\\
\le2\L(V)(1+2\sqrt{d}L(V)t\Dlt t)\Dlt te^{t\max(1,2\sqrt{d}L(V))}\frac{M_1(t)}{\hb}
\\
+2\L(V)^2t\Dlt t\max(1,\Dlt t)e^{6t\max(1,\sqrt{d}L(V))}\frac{M_2(t)}{\hb^2}&\,.
\ea
\ee

\subsection{Specializing to the case where $A$ is a Weyl operator}\lb{SS-Step7}

%%%%%%%%%%%%%%%%%%%%%%%%%%%%%%%%%%%%%%%%%%%%%%%%%%%%%%%%%%%%%%%%%%%%%%%%%%%%%%%%%%%%%%%%%%%%%%%%%%%%%%%%%%%%%

In order to finish the proof of Theorem \ref{T-Main}, we restrict our attention to a convenient class of test operators $A$, for which the quantities $M_1(t)/\hb$ and $M_2(t)/\hb^2$ are bounded as $\hb\to 0$.

We first recall the definition of a Weyl operator on $\fH$. For each $a\in\cS'(\bR^d\times\bR^d)$, one defines a linear map $\Op^W_\hb[a]:\,\cS(\bR^d)\mapsto\cS'(\bR^d)$ by the following duality formula:
$$
\la\Op^W_\hb[a]\psi,\overline{\phi}\ra_{\cS'(\bR^d),\cS(\bR^d)}:=\la a,W_\hb[|\phi\ra\la\psi|]\ra_{\cS'(\bR^d),\cS(\bR^d)}\,.
$$
Elementary computations show that
$$
\tfrac{i}\hb[x^\nu,\Op^W_\hb[a]]=-\Op^W_\hb[\d_{\xi^\nu}a]\,,\quad\tfrac{i}\hb[-i\hb\d_{x^\nu},\Op^W_\hb[a]]=\Op^W_\hb[\d_{x^\nu}a]\,.
$$
Boulkhemair's improvement \cite{Boulkhem} of the Calderon-Vaillancourt theorem states that, for each integer $d\ge 1$, there exists $\g_d>0$ such that, for each $a\in\cS'(\bR^d\times\bR^d)$ satisfying 
the condition $\d_x^\a\d_\xi^\b a\in L^\infty(\bR^d\times\bR^d)$ whenever $|\a|$ and $|\b|\le[d/2]+1$, 
$$
\max_{|\a|,|\b|\le[d/2]+1}\|\d_x^\a\d_\xi^\b a\|_{L^\infty(\bR^d\times\bR^d)}\le 1\implies\|\Op^W_\hb[a]\|\le\g_d.
$$

Therefore, choosing $A=\Op^W_\hb[a]$ implies that
$$
\frac{M_1(t)}\hb\le 2\g_d\max_{|\a|,|\b|\le[d/2]+2}\|\d_x^\a\d_\xi^\b a\|_{L^\infty(\bR^d\times\bR^d)}\,,
$$
and
$$
\frac{M_2(t)}{\hb^2}\le 3\g_d\max_{|\a|,|\b|\le[d/2]+3}\|\d_x^\a\d_\xi^\b a\|_{L^\infty(\bR^d\times\bR^d)}\,,
$$
so that
$$
\ba
\left|\Tr_\fH\left((\bE\wtilde R_\indc(t)-R_\indc(t))\Op^W_\hb[a]\right)\right|
\\
\le 2\L(V)(1+2\sqrt{d}L(V)t\Dlt t)\Dlt te^{t\max(1,2\sqrt{d}L(V))}\frac{M_1(t)}{\hb}
\\
+2\L(V)^2t\Dlt t\max(1,\Dlt t)e^{6t\max(1,\sqrt{d}L(V))}\frac{M_2(t)}{\hb^2}
\\
\le 2\Dlt te^{6t\max(1,\sqrt{d}L(V))}\L(V)(2+3t\L(V)\max(1,\Dlt t)+4\sqrt{d}L(V)t\Dlt t)
\\
\times\g_d\max_{|\a|,|\b|\le[d/2]+3}\|\d_x^\a\d_\xi^\b a\|_{L^\infty(\bR^d\times\bR^d)}&\,.
\ea
$$

Since $\bE\wtilde R_\indc(t)-R_\indc(t)$ is a self-adjoint element of $\cL^1(\fH)$, there exists a complete orthonormal sequence $\phi_k$ in $\fH$ such that
$$
\bE\wtilde R_\indc(t)-R_\indc(t)=\sum_{k\ge 1}\rho_k|\phi_k\ra\phi_k|
$$
with 
$$
\rho_k\in\bR\quad\text{ and }\quad\sum_{k\ge 1}|\rho_k|<\infty\,.
$$
Therefore, if $a\in\cS(\bR^d\times\bR^d)$ satisfies
$$
\max_{|\a|,|\b|\le[d/2]+3}\|\d_x^\a\d_\xi^\b a\|_{L^\infty(\bR^d\times\bR^d)}<\infty\,,
$$
then
$$
\ba
\La W_\hb[\bE\wtilde R_\indc(t)]-W_\hb[R_\indc(t)],a\Ra_{\cS'(\bR^d),\cS(\bR^d)}=&\sum_{k\ge 1}\rho_k\la W_\hb[|\phi_k\ra\la\phi_k|],a\ra_{\cS'(\bR^d),\cS(\bR^d)}
\\
=&\Tr_\fH\left((\bE\wtilde R_\indc(t)-R_\indc(t))\Op^W_\hb[a]\right)\,.
\ea
$$
In other words
$$
\ba
|||W_\hb[\bE\wtilde R_\indc(t)]-W_\hb[R_\indc(t)]|||_{-[d/2]-3}
\\
\le 2\g_d\Dlt te^{6t\max(1,\sqrt{d}L(V))}\L(V)(2+3t\L(V)\max(1,\Dlt t)+4\sqrt{d}L(V)t\Dlt t)&\,.
\ea
$$

%%%%%%%%%%%%%%%%%%%%%%%%%%%%%%%%%%%%%%%%%%%%%%%%%%%%%%%%%%%%%%%%%%%%%%%%%%%%%%%%%%%%%%%%%%%%%%%%%%%%%%%%%%%%%

\section{Remarks on Theorem \ref{T-Main} and its Proof}

%%%%%%%%%%%%%%%%%%%%%%%%%%%%%%%%%%%%%%%%%%%%%%%%%%%%%%%%%%%%%%%%%%%%%%%%%%%%%%%%%%%%%%%%%%%%%%%%%%%%%%%%%%%%%

\subsection{Metrizing the set of density operators} 

%%%%%%%%%%%%%%%%%%%%%%%%%%%%%%%%%%%%%%%%%%%%%%%%%%%%%%%%%%%%%%%%%%%%%%%%%%%%%%%%%%%%%%%%%%%%%%%%%%%%%%%%%%%%%

Our analysis in the present paper can be expressed in terms of the following metric on the set of density operators on the Hilbert space $\fH$.

\begin{Def}\lb{D-Defdhb}
For $R,S\in\cD(\fH)$, set $D:=-i\d$ and 
$$
\ba
d_\hb(R,S)\!:=\!\sup\left\{|\Tr_\fH((R\!-\!S)A)|\,\,\left|\,\,\ba{}&A\in\cL(\fH)\text{ and for all }1\le\mu,\nu\le d
\\
&\hb\|[x^\mu\!,\!A]\|\!+\!\hb\|[\hb D_\mu,A]\|+\|[x^\nu\!,\![x^\mu\!,\!A]]\|
\\
&\!+\!\|[\!\hb D_\nu,[x^\mu\!,\!A]]\|\!+\!\|[\!\hb D_\nu,[\hb D_\mu,A]]\|\!\le\!5\hb^2\ea\right.\right\}\,.
\ea
$$
\end{Def}

The distance is analogous to several distances introduced earlier in the literature. The first is obviously the Monge-Kantorovich(-Rubinstein) distance, also referred to as the Wasserstein distance of exponent $1$:
see formula (7.1) in chapter 7 of \cite{Villani}. However, the Monge-Kantorovich distance is defined on the set of Borel probability measures on the Euclidean space $\bR^d$, and not on density operators on $\fH$.

An analogue of the Monge-Kantorovich distance has been proposed by Connes on the set of states on a $C^*$-algebra endowed with an unbounded Fredholm module: see Proposition 4 in \cite{ConnesFred}, or
\S 5 in the Introduction and \S 1 in chapter 6 of \cite{ConnesNC}. See also the review paper \cite{Martinetti} for a more thorough discussion of this distance. However, the Connes distance is the noncommutative 
analogue of a Riemannian metric on compact spin manifold, as explained in Proposition 1 of \cite{ConnesFred} --- see also formula (2.9) in \cite{Martinetti} which does not involve a spin structure --- and not on a 
phase space, i.e. not on a cotangent bundle. The analogue of the Monge-Kantorovich or Wasserstein distance of exponent $1$, or of Connes' distance in our setting would be
$$
MK_1^\hb(R,S):=\sup_{A\in\cL(\fH)\atop\max_{1\le\mu\le d}(\|[x^\mu,A]\|,\|[\hb D_\mu,A])\le\hb}|\Tr_\fH((R\!-\!S)A)|\,.
$$
Since the correspondence principle associates the Poisson bracket $\{\cdot,\cdot\}$ to $\frac{i}\hb[\cdot,\cdot]$, and since $\{x^\mu,\cdot\}=-\d_{\xi^\mu}$ while $\{\xi^\mu,\cdot\}=-\d_{x^\mu}$, the constraint
$\|[\hb D_\mu,A]\le\hb$ corresponds to Lipschitz continuity in the position variable $x^\mu$, while the constraint $\|[x^\mu,A]\|\le\hb$ corresponds to Lipschitz continuity in the momentum variable $\xi^\mu$.
The distance $d_N^T$ used in \cite{FGTPaulEmpir} to prove the uniformity in the Planck constant $\hb$ of the mean-field limit in quantum mechanics (see formula (43) and Theorem 1.1 in \cite{FGTPaulEmpir})
is essentially based on the same idea as $MK_1^\hb$. 

While $d_\hb$ is also based on the same idea as $MK_1^\hb$, the constraints entering its definition uses iterated commutators because of the specifics of the proof of convergence for the random batch method.
More precisely, the need for iterated commutators comes from the key step using the independence of the reshuffling permutations $\si_1,\si_2,\ldots$ leading to \eqref{WeakIdent} and the quantity $\Dlt B_N(s,j)$.
Indeed, estimating $\Dlt B_N(s,j)$ systematically involves iterated brackets, as shown in \eqref{FlaDeltaB} and the subsequent formulas.

\smallskip
The main properties of $d_\hb$ are summarized in the following proposition.

\begin{Prop}\lb{P-dhb}
For each $\hb>0$, the function $d_\hb$ is defined on $\cD(\fH)\times\cD(\fH)$ and takes its values in $[0,+\infty]$. Moreover

\noindent
(i) the function $d_\hb$ is an extended metric on $\cD(\fH)$: it is symmetric in its two arguments, satisfies the triangle inequality, and separates points in $\cD(\fH)$;

\noindent
(ii) there exists $\g_d>0$ (depending only on the space dimension $d$) such that, for all $R,S\in\cD(\fH)$
$$
|||W_\hb[R]-W_\hb[S]|||_{-[d/2]-3}\le\g_dd_\hb(R,S)\,.
$$
\end{Prop}

\begin{proof}
That $d_\hb$ is symmetric in its arguments and satisfies the triangle inequality is obvious from the definition. That $d_\hb$ separates points in $\cD(\fH)$ follows from (ii). Indeed
$$
d_\hb(R,S)=0\implies|||W_\hb[R]-W_\hb[S]|||_{-[d/2]-3}=0\,,
$$
so that
$$
\iint_{\bR^d\times\bR^d}W_\hb[R-S](x,\xi)\overline{a(x,\xi)}dxd\xi=0
$$
for all $a\equiv a(x,\xi)\in\cS(\bR^d\times\bR^d)$ such that $\|\d_x^\a\d_\xi^\b a\|_{L^\infty(\bR^{2d})}\le 1$ for all multi-indices $\a,\b\in\bN^d$ such that $|\a|,|\b|\le[d/2]+3$, and therefore, by homogeneity and density of $\cS(\bR^d\times\bR^d)$
in $L^2(\bR^d\times\bR^d)$, for all $a\equiv a(x,\xi)\in L^2(\bR^d\times\bR^d)$. Hence $W_\hb[R]=W_\hb[S]$ in $L^2(\bR^d\times\bR^d)$, and since the Fourier transform is invertible on $L^2(\bR^d\times\bR^d)$, this implies that $R$ and $S$ 
have integral kernels a.e. equal, so that $R=S$. This proves Property (i) taking Property (ii) for granted.

The proof of Property (ii) is essentially a repetition of the last step in the proof of Theorem \ref{T-Main} (section \ref{SS-Step7}). Indeed
$$
\ba
|||W_\hb[R]-W_\hb[S]|||_{-[d/2]-3}
\\
=\sup\left\{\left|\int_{\bR^d\times\bR^d}W_\hb[R-S](x,\xi)(x,\xi)dxd\xi\right|\text{ s.t. }\max_{|\a|,|\b|\le[d/2]+3}|\d_x^\a\d_\xi^\b a(x,\xi)|=1\right\}
\\
=\sup\left\{\left|\Tr_\fH\left((R-S)\Op_\hb^W[a]\right)\right|\text{ s.t. }\max_{|\a|,|\b|\le[d/2]+3}|\d_x^\a\d_\xi^\b a(x,\xi)|=1\right\}
\\
\le\g_d d_\hb(R,S)&\,,
\ea
$$
since
$$
\ba
\max_{|\a|,|\b|\le[d/2]+3}|\d_x^\a\d_\xi^\b a(x,\xi)|=1\implies&\hb\|[x^\mu\!,\!A]\|\!+\!\hb\|[\hb D_\mu,A]\|+\|[x^\nu\!,\![x^\mu\!,\!A]]\|
\\
&\!+\!\|[\!\hb D_\nu,[x^\mu\!,\!A]]\|\!+\!\|[\!\hb D_\nu,[\hb D_\mu,A]]\|\!\le\!5\g_d\hb^2
\ea
$$
by Boulkhemair's variant \cite{Boulkhem} of the Calderon-Vaillancourt theorem.
\end{proof}

\smallskip
The error estimate in Theorem \ref{T-Main} could have been couched in terms of the distance $d_\hb$. Indeed, the inequality \eqref{LastIneq} at the end of the penultimate step in the proof of Theorem \ref{T-Main} can be recast as
\be\lb{LastIneq2}
\ba
d_\hb(\bE\wtilde R_\indc(t),R_\indc(t))\le&10\L(V)^2t\Dlt t\max(1,\Dlt t)e^{6t\max(1,\sqrt{d}L(V))}
\\
&+10\L(V)(1+2\sqrt{d}L(V)t\Dlt t)\Dlt te^{t\max(1,2\sqrt{d}L(V))}\,.
\ea
\ee
Up to unessential modifications in the constants, the error estimate in Theorem \ref{T-Main} is a consequence of this inequality and Proposition \ref{P-dhb} (ii).

\subsection{On the choice of $|||\cdot|||_{-[d/2]-3}$ or $d_\hb$ to express the error bound} 

%%%%%%%%%%%%%%%%%%%%%%%%%%%%%%%%%%%%%%%%%%%%%%%%%%%%%%%%%%%%%%%%%%%%%%%%%%%%%%%%%%%%%%%%%%%%%%%%%%%%%%%%%%%%%

The idea of using the metric $d_\hb$ presented in the previous section might seem strange. One might find it more natural to use more traditional metrics on density operators, such as the trace norm, for instance. Indeed, for all
$t,s\in\bR$, the map $\cU(t,s)$ is an isometry for the trace norm, because the map $U(t,s)$ is a unitary operator on $\fH_N$. 

Estimating the difference $\bE\wtilde R_\indc(t)-R_\indc(t)$ in trace norm can be done along the same line as in section \ref{S-Main}. Although this is not the simplest route to obtaining this estimate, it will be easier to compare the 
inequalities at each step in this estimate with the ones using $d_\hb$. Indeed
$$
\|\bE\wtilde R_\indc(t)-R_\indc(t)\|_1=\sup_{\|A\|\le 1}|\Tr((\bE\wtilde R_\indc(t)-R_\indc(t))A)|\,,
$$
and using \eqref{WeakIdent} shows that
$$
\ba
\|\bE\wtilde R_\indc(t)-R_\indc(t)\|_1\le\frac1\hb\int_{[\frac{t}{\Dlt t}]\Dlt t}^t\sup_{\|A\|\le 1}\left|\Tr_{\fH_N}\left(R(s)\bE\left[\sum_{1\le m<n\le N}\cV_{mn},B_N(s)\right]\right)\right|ds
\\
+\frac1\hb\sum_{j=1}^{[\frac{t}{\Dlt t}]}\int_{(j-1)\Dlt t}^{j\Dlt t}\sup_{\|A\|\le 1}\left|\Tr_{\fH_N}\Bigg(R(s)\bE\Bigg[\sum_{1\le m<n\le N}\cV_{mn},\Dlt B_N(s,j)\Bigg]\Bigg)\right|ds&\,,
\ea
$$
with the notation
$$
\cV_{mn}:=\left(\bT_s(m,n)-\tfrac1{N-1}\right)V_{mn}\,.
$$
Then
$$
\|A\|\le 1\implies\left\|\frac1N\sum_{k=1}^NJ_kA\right\|\le 1\implies\|B_N(s)\|\le 1
$$
for all $s\in[0,t]$, and
\be\lb{[cVB]<}
\ba
\left|\Tr_{\fH_N}\left(R(s)\bE\left[\sum_{1\le m<n\le N}\cV_{mn},B_N(s)\right]\right)\right|
\\
\le\sum_{1\le m<n\le N}\left|\Tr_{\fH_N}\left(R(s)[\bE\cV_{mn},B_N(s)]\right)\right|
\\
\le 2N\|V\|_{L^\infty(\bR^d)}\|R(s)\|_1\|B_N(s)\|=2N\|V\|_{L^\infty(\bR^d)}&\,,
\ea
\ee
since
$$
\ba
\bT_s(m,n)=0&\implies\|\cV_{mn}\|=\frac{1}{N-1}\|V\|_{L^\infty(\bR^d)}
\\
\bT_s(m,n)=1&\implies\|\cV_{mn}\|=\frac{N-2}{N-1}\|V\|_{L^\infty(\bR^d)}\,.
\ea
$$
Likewise
\be\lb{[cVDB]<}
\ba
\left|\Tr_{\fH_N}\Bigg(R(s)\bE\Bigg[\sum_{1\le m<n\le N}\cV_{mn},\Dlt B_N(s,j)\Bigg]\Bigg)\right|
\\
\le\sum_{1\le m<n\le N}\left|\Tr_{\fH_N}(R(s)\bE[\cV_{mn},\Dlt B_N(s,j)])\right|
\\
\le 2N\|V\|_{L^\infty(\bR^d)}\|\Dlt B_N(s,j)\|&\,,
\ea
\ee
and \eqref{FlaDeltaB} implies that
\be\lb{|DB|<}
\Dlt B_N(s,j)\|\le\frac1\hb\Dlt tN\|V\|_{L^\infty(\bR^d)}\,.
\ee
Putting all these estimates together results in the upper bound
$$
\ba
\|\bE\wtilde R_\indc(t)-R_\indc(t)\|_1\le&\frac2\hb\Dlt t\cdot N\|V\|_{L^\infty(\bR^d)}
\\
&+\frac2\hb\left[\frac{t}{\Dlt t}\right]\Dlt tN\|V\|_{L^\infty(\bR^d)}\cdot\frac1\hb\Dlt tN\|V\|_{L^\infty(\bR^d)}
\\
\le&\frac{2N}\hb\Dlt t\|V\|_{L^\infty(\bR^d)}\left(1+\frac{Nt}\hb\|V\|_{L^\infty(\bR^d)}\right)\,,
\ea
$$
which is neither uniform as $N\to\infty$ nor as $\hbar\to 0$, and therefore satisfies neither of our requirements (a) and (b) at the end of section \ref{S-Intro}.

It is instructive to compare the rather naive estimates above with the more subtle corresponding estimate in the proof of Theorem \ref{T-Main}. 

For instance, comparing \eqref{[cVB]<} with \eqref{[VB]<}, or \eqref{[cVDB]<} with \eqref{[VDB]<} shows clearly that \eqref{[VB]<} or \eqref{[VDB]<} involve \textit{only} the commutators with the variables $x_k^\mu$ that are 
present in the potential $V(x_m-x_n)$, i.e. only the two values $k=m$ or $k=n$. This key observation is at the core of section \ref{SS-SCPotEst}.

When summing over all possible pairs $m,n$ either with $m$ and $n$ in the same batch, or over all $m,n$ with the coupling constant $1/N$, one arrives at the bound \eqref{Last[]}, which does not involve the $N$ factor that is 
present in \eqref{[cVB]<}. One might suspect that this $N$ factor is hidden in the summation over $m=1,\ldots,N$ in the right hand side of \eqref{Last[]}, but in fact this summation is included in the definition \eqref{DefM1}, and 
the bounds \eqref{M1s<} and \eqref{M1t} make it clear that no $N$ factor can arise in this way. The key observation is obviously the bound \eqref{M1t} which does not involve $N$, since 
$$
B_N(t)=\frac1N\sum_{k=1}J_kA
$$
and $[x_m^\mu,J_kA]=[-i\hb\d_{x_m^\mu},J_kA]=0$ for all $\mu=1,\ldots,d$ and all $m=1,\ldots,N$ unless $m=k$. Finally, the definition of $d_\hb$ implies that the test operator $A$ satisfies both $\|[x_m^\mu,A]\|=O(\hb)$ and 
$\|[-i\hb\d_{x_m^\mu},J_kA]\|=O(\hb)$ for all $m=1,\ldots,N$ and $\mu=1,\ldots,d$, so that $M_1(t)=O(\hb)$. This nice bound (small as $\hb\to 0$, independent of $N$) is propagated by the random batch dynamics as explained 
in section \ref{SS-1stSC}. As a result, the bound \eqref{Last[]} does not involve the unpleasant $N$ factor in \eqref{[cVB]<}, and the fact that $M_1(t)=O(\hb)$ offsets the $1/\hb$ factor multiplying the last time integral on the right
hand side of \eqref{WeakIdent}, at variance with the naive estimate above.

The same advantages of using the $d_\hb$ metric instead of the trace norm are observed in the treatment of the ``generic'' term, i.e. the integral over the time interval $((j-1)\Dlt t,j\Dlt t)$ on the right hand side of \eqref{WeakIdent}. 
The naive estimate above, i.e. \eqref{[cVDB]<} and \eqref{|DB|<}, lead to an even more disastrous bound of order $N^2/\hb$ (there is one factor $N$ that comes for the same reason as in \eqref{[cVB]<}, and an additional factor 
$N/\hb$ which comes from the estimate \eqref{|DB|<} in operator norm based on Duhamel formula for $\Dlt B_N(s,j)$). Instead, one repeats with $\Dlt B_N(s,j)$ the same argument as in the treatment of \eqref{Last[]}. Since the
term $\Dlt B_N(s,j)$ is itself the time integral of a commutator involving the random batch potential, the same rarefaction in the relevant commutators $\|[x_k^\ka,B]\|$ used to control $\|[V_{mn},B_N(\tau)]\|$ is observed, except 
that one needs bounds for commutators of the form $\|[x_k^\ka,\Dlt B_N(s,j)]\|$ and not $\|\Dlt B_N(s,j)\|$ itself. This is the reason why we need to control iterated brackets of order $2$, which is done in section \ref{SS-2ndSC}. 
The bound \eqref{Gen[]} and the inequality \eqref{[xmmuDltB]} show that everything can be controlled in terms of the quantity $M_2(\tau)$ defined in \eqref{DefM2}. Here again, one might suspect that the summation over $m,n$ 
in \eqref{DefM2} would produce the same unpleasant $N^2$ factor that appears when using \eqref{[cVDB]<} and \eqref{|DB|<}, but the bounds \eqref{M2s<} and \eqref{M2t} clearly show that this is not the case. 

Eventually $M_2(\tau)=O(\hb^2)$ (uniformly in $N$) because of the choice of the test operator $A$ in the definition of $d_\hb$: here the key estimate is \eqref{RarefBt}, which explains why the sum of $N^2$ terms in \eqref{DefM2}
produces a quantity that is bounded uniformly in $N$. That $M_2(t)=O(\hb^2)$ follows from the condition on $A$ in the definition of the metric $d_\hb$, and this offsets the $1/\hb^2$ due to the integral over the time interval 
$((j-1)\Dlt t,j\Dlt t)$ in \eqref{WeakIdent}, and to the additional time integral in the Duhamel formula \eqref{FlaDeltaB} for $\Dlt B_n(j,s)$.

Summarizing the discussion above, the metric $d_\hb$ is especially designed in order to handle both the large $N$ and the small $\hb$ issues, i.e. requirements (a) and (b) in the introduction.

%%%%%%%%%%%%%%%%%%%%%%%%%%%%%%%%%%%%%%%%%%%%%%%%%%%%%%%%%%%%%%%%%%%%%%%%%%%%%%%%%%%%%%%%%%%%%%%%%%%%%%%%%%%%%

\section{Conclusion and Perspectives}

%%%%%%%%%%%%%%%%%%%%%%%%%%%%%%%%%%%%%%%%%%%%%%%%%%%%%%%%%%%%%%%%%%%%%%%%%%%%%%%%%%%%%%%%%%%%%%%%%%%%%%%%%%%%%

We have given an error estimate (Theorem \ref{T-Main}) for the simplest imaginable random batch method applied to the quantum dynamics of $N$ identical particles. This error estimate has the advantage of being independent
of the particle number $N$ and of the Planck constant $\hb$ (or more precisely of the ratio of the Planck constant to the typical action of one particle in the system). On the other hand, the final estimate is stated in terms of some
dual (negative) Sobolev type norm on the difference between the expected single body reduced density operators propagated from the same initial state by the random batch dynamics and by the $N$-particle dynamics. For the 
sake of simplicity, we have restricted our attention to the simplest case of batches of two particles only. 

The main new mathematical ingredient in the proof is the use of the metric $d_\hb$ (see Definition \ref{D-Defdhb} for the definition of this object, and Proposition \ref{P-dhb} for its basic properties), which is especially tailored 
to handle at the same time the difficulties pertaining to the small $\hb$ regime (classical limit), and those pertaining to the large $N$ regime (mean-field limit). The final statement (Theorem \ref{T-Main}) of the error estimate
does not involve the metric $d_\hb$, but is couched in terms of the Wigner transforms \cite{LionsPaul} of the $N$-body and random batch density operators, a mathematical object which is familiar in the context of quantum 
dynamics.

Several extensions of this result should be easily obtained with the mathematical tools used in the present paper. First one can obviously consider batches of $p>2$ particles; the error analysis is expected to be similar. Also, 
in practice, the random batch dynamics \eqref{RBSchro} or \eqref{RBvN} is further approximated by some convenient numerical scheme. Of course, the numerical schemes used on \eqref{RBSchro} or \eqref{RBvN} should 
satisfy the same requirements (a) and (b) (uniform convergence in $N$ and in $\hb$) listed in the introduction. For instance, time-splitting strategies for quantum dynamics converge uniformly in $\hb$ (see \cite{BaoJinMarko} 
and \cite{FGSJinTPaul}), and could be used together with random batch strategies. The numerical treatment of the space variable $x$, however, seems much more challenging.

\textbf{Acknowledgements.} The work of Shi Jin was partly supported by NSFC grants No. 11871297 and No. 31571071. We thank E. Moulines for kindly indicating several references on stochastic approximation.
%, and N. Lerner for additional information on the Calderon-Vaillancourt theorem.

%%%%%%%%%%%%%%%%%%%%%%%%%%%%%%%%%%%%%%%%%%%%%%%%%%%%%%%%%%%%%%%%%%%%%%%%%%%%%%%%%%%%%%%%%%%%%%%%%%%%%%%%%%%%%

%%%%%%%%%%%%%%%%%%%%%%%%%%%%%%%%%%%%%%%%%%%%%%%%%%%%%%%%%%%%%%%%%%%%%%%%%%%%%%%%%%%%%%%%%%%%%%%%%%%%%%%%%%%%%


\begin{thebibliography}{99}

\bibitem{BachMoulines}
F. Bach, E. Moulines:
\textit{Non-Asymptotic Analysis of Stochastic Approximation Algorithms for Machine Learning},
in ``Proc. Advances in Neural Information Processing Systems (NIPS)'', Granada, Spain, 2011, pp. 451--459.

\bibitem{BaoJinMarko}
W. Bao, S. Jin, P.A. Markowich: 
\textit{On time-splitting spectral approximations for the Schr\"odinger equation in the semiclassical regime}, 
J. Comp. Phys. \textbf{175} (2002), 487--524.

\bibitem{Benaim}
M. Benaim:
\textit{Dynamics of stochastic approximation algorithms}
in ``S\'eminaire de Probabilit\'es'', Strasbourg, tome 33, (1999), pp. 1--68,
Springer-Verlag, Berlin Heidelberg New York, 1999.

\bibitem{BeneJakPortaSaffSchlein}
N. Benedikter, V. Jaksic, M. Porta, C. Saffirio, B. Schlein:
\textit{Mean-field evolution of fermionic mixed states},
Commun. Pure Appl. Math. \textbf{69} (2016), 2250--2303.

\bibitem{BenePortaSaffSchlein} 
N. Benedikter, M. Porta, C. Saffirio, B. Schlein:
\textit{From the Hartree dynamics to the Vlasov equation}, 
Arch. Ration. Mech. Anal. \textbf{221} (2016), 273--334.

\bibitem{BenePortaSchlein} 
N. Benedikter, M. Porta, B. Schlein:
\textit{Mean-field evolution of Fermionic systems}, 
Commun. Math. Phys. \textbf{331} (2014), 1087--1131.

\bibitem{Boulkhem}
A. Boulkhemair:
\textit{$L^2$ estimates for Weyl quantization},
J. Functional Anal. \textbf{165} (1999), 173--204.

\bibitem{ConnesFred}
A. Connes:
\textit{Compact metric spaces, Fredholm modules, and hyperfiniteness},
Ergod. Th. and Dynam. Sys. \textbf{9} (1989), 207--220.

\bibitem{ConnesNC}
A. Connes:
``Noncommutative Geometry'',
Academic Press, Inc., San Diego, CA, 1994.

\bibitem{Durstenfeld}
R. Durstenfeld:
\textit{Algorithm 235: random permutation}, 
Commun. of the ACM, \textbf{7} (1964), 420.

\bibitem{FGSJinTPaul}
F. Golse, S. Jin, T. Paul:
\textit{On the Convergence of Time Splitting Methods for Quantum Dynamics in teh Semiclassical Regime},
preprint {\tt arXivarXiv:1906.03546 [math.NA]}.

\bibitem{FGMouPaul}
F. Golse, C. Mouhot, T. Paul:
\textit{On the Mean Field and Classical Limits of Quantum Mechanics},
Commun. Math. Phys. \textbf{343} (2016), 165--205.

\bibitem{FGTPaulARMA2017}
F. Golse, T. Paul:
\textit{The Schr\"odinger Equation in the Mean-Field and Semiclassical Regime},
Arch. Rational Mech. Anal. \textbf{223} (2017), 57--94.

\bibitem{FGTPaulCRAS}
F. Golse, T. Paul:
\textit{Wave Packets and the Quadratic Monge-Kantorovich Distance in Quantum Mechanics},
C. R. Acad. Sci. Paris, S\'er. I \textbf{356} (2018), 177--197.

\bibitem{FGTPaulEmpir}
F. Golse, T. Paul:
\textit{Empirical Measures and Quantum Mechanics: Applications to the Mean-Field Limit},
Commun. Math. Phys. \textbf{369} (2019), 1021--1053.

\bibitem{SJinLLiJGLiu}
S. Jin, L. Li, J.-G. Liu:
\textit{Random Batch Method for Interacting Particle Systems},
J. Comput. Phys., \textbf{400} (2020), 108877.

\bibitem{JMS}
S. Jin, P.A. Markowich, C. Sparber,
\textit{Mathematical and computational methods for semiclassical Schr{\"o}dinger equations},
 Acta Numerica, \textbf{20} (2011), 121-209.
 

\bibitem{Kushner}
H.J. Kushner, G.G. Yin:
``Stochastic Approximation and Recursive Algorithms and Applications'', 2nd edition,
Springer Verlag, New York, 2003.

\bibitem{LionsPaul}
P.-L. Lions, T. Paul: 
\textit{Sur les mesures de Wigner},
Rev. Math. Iberoam. \textbf{9} (1993), 553--618.

\bibitem{Martinetti}
P. Martinetti:
\textit{From Monge to Higgs: a survey of distance computations in noncommutative geometry}. 
In ``Noncommutative geometry and optimal transport'', 1--46, 
Contemp. Math., 676, Amer. Math. Soc., Providence, RI, 2016.

\bibitem{Nemirovski}
A. Nemirovski, A. Juditsky, G. Lan, A.Shapiro:
\textit{Robust stochastic approximation approach to stochastic programming},
SIAM J. Optimization \textbf{19} (2009), 1574--1609.

\bibitem{Villani}
C. Villani:
``Topics in Optimal Transportation'',
American Mathematical Soc, Providence (RI) (2003)

\bibitem{YingYuanVlaskiSayed}
B. Ying, K. Yuan, S. Vlaski, A. H. Sayed:
\textit{Stochastic Learning under Random Reshuffling with Constant Step-sizes}\,,
preprint {\tt arXiv 1803.07964 [cs.LG]}.

\end{thebibliography}
\end{document}